\documentclass[reqno,twoside]{amsart}
\usepackage{amsfonts}
\usepackage{amssymb}
\usepackage{amsmath}
\usepackage[english]{babel}
\usepackage{hyperref}
\usepackage{microtype}
\usepackage[utf8]{inputenc}
\usepackage{xcolor}

\newcommand{\func}[1]{\operatorname{#1}}

\DisableLigatures{encoding = *, family = * }

\newtheorem{theorem}{Theorem}
\newtheorem{definition}[theorem]{Definition}

\newtheorem{lemma}[theorem]{Lemma}

\newtheorem{proposition}[theorem]{Proposition}
\newtheorem{remark}[theorem]{Remark}

\numberwithin{equation}{section}
\numberwithin{theorem}{section}

\newcommand{\C}{\mathcal{C}}

\title[Multiplicity of solutions]{Multiplicity of solutions for fractional $q(.)$-Laplacian equations}

\author{Abita Rahmoune\textsuperscript{\,$\dagger$}}
\address{\textsuperscript{$\dagger$}\,Laboratory of Pure and Applied Mathematics, University of Laghouat, P.O. Box 37G, Laghouat 03000, Algeria}
\email{abitarahmoune@yahoo.fr}

\author{Umberto Biccari\textsuperscript{\,$\ast$}}
\address{\textsuperscript{$\ast$}\, [1] Chair of Computational Mathematics, Fundaci\'on Deusto, Avenida de las Universidades 24, 48007 Bilbao, Basque Country, Spain}
\address{[2]\,Facultad de Ingenier\'ia, Universidad de Deusto, Avenida de las Universidades 24, 48007 Bilbao, Basque Country, Spain.}
\email{umberto.biccari@deusto.es, u.biccari@gmail.com}

\thanks{This project has received funding from the European Research Council (ERC) under the European Union's Horizon 2020 research and innovation programme (grant agreement NO: 694126-DyCon). The work of UB is partially supported by the Air Force Office of Scientific Research (AFOSR) under Award NO: FA9550-18-1-0242, by the Grant MTM2017-92996-C2-1-R COSNET of MINECO (Spain) and by the Elkartek grant KK-2020/00091 CONVADP of the Basque government}


\begin{document}

\begin{abstract}
In this paper, we deal with the following elliptic type problem 
\begin{displaymath}
	\begin{cases}
		(-\Delta)_{q(.)}^{s(.)}u + \lambda Vu = \alpha \left\vert u\right\vert^{p(.)-2}u+\beta \left\vert u\right\vert^{k(.)-2}u & \text{ in }\Omega,
		\\[7pt]
		u =0 & \text{ in }\mathbb{R}^{n}\backslash \Omega ,
	\end{cases}
\end{displaymath}
where $q(.):\overline{\Omega}\times \overline{\Omega}\rightarrow \mathbb{R}$ is a measurable function and $s(.):\mathbb{R}^n\times \mathbb{R}^n\rightarrow (0,1)$ is a continuous function, $n>q(x,y)s(x,y)$ for all $(x,y)\in \Omega \times \Omega $, $(-\Delta)_{q(.)}^{s(.)}$ is the variable-order fractional Laplace operator, and $V$ is a positive continuous potential. Using the mountain pass category theorem and Ekeland's variational principle, we obtain the existence of a least two different solutions for all $\lambda>0$. Besides, we prove that these solutions converge to two of the infinitely many solutions of a limit problem as $\lambda \rightarrow +\infty $. 
\end{abstract}

\maketitle

\section{Introduction}

In recent years, many authors have paid attention to the study of nonlocal fractional operators and related fractional differential equations. This is partially due to the large employment of these operators to model several phenomena such as ultra-relativistic limits of quantum mechanics, phase transition, population dynamics, minimal surfaces and game theory.

In this paper, we deal with the following elliptic equation for the fractional Laplace operator with variable order derivative involving
variable exponent nonlinearities:
\begin{equation}\label{mainProblem}
	\begin{cases}
		(-\Delta)_{q(.)}^{s(.)}u +\lambda V u = \alpha \left\vert u\right\vert^{p(.)-2}u + \beta \left\vert u\right\vert^{k(.)-2}u & \text{ in }\Omega,
		\\[7pt]
		u = 0 & \text{ in }\mathbb{R}^n\backslash \Omega.
	\end{cases}
\end{equation}

In \eqref{mainProblem}, for all $q(.) :\overline{\Omega }\times \overline{\Omega}\rightarrow (1,+\infty)$ measurable and $s(.):\mathbb{R}^{n}\times \mathbb{R}^{n}\rightarrow (0,1)$ continuous, with $n>q\left( x,y\right) s(x,y)$ for all $(x,y)\in \Omega \times \Omega$, we denote by $(-\Delta)_{q(.)}^{s(.)}$ the variable-order fractional $q-$Laplace operator which is defined for any $\varphi \in C_0^{\infty}(\Omega)$ as
\begin{align}\label{fl}
	&(-\Delta)_{q(.)}^{s(.)}\varphi(x) = 2\lim_{\varepsilon \rightarrow 0^+}\int_{\mathbb{R}^n\backslash B_{\varepsilon}(x)} \frac{\left\vert\varphi(x) -\varphi(y) \right\vert^{q(x,y)-2}\left(\varphi(x) -\varphi(y)\right)}{\left\vert x-y\right\vert^{n+q(x,y)s(x,y)}}\,\mathrm{d}y,
	\\
	&\text{for all }x\in \mathbb{R}^{n},
\end{align}
where $B_{\varepsilon }\left( x\right)$ indicates the ball of radius $\varepsilon >0$ centered at $x\in \mathbb{R}^{n}$. Moreover, $V:\Omega \rightarrow \lbrack 0,+\infty )$ is a continuous function and $\alpha$, $\beta$, $\lambda>0$ are positive parameters. Finally, the variable exponents $k(.)$ and $p(.)$ of the nonlinear terms are given measurable functions on $\Omega$.

The terminology \textit{variable-order} fractional Laplace operator indicates that $s(.)$ and $q(.)$ are functions and not real numbers. This operator is then a generalization of the fractional Laplacian $(-\Delta)^s$, which corresponds to $q(.)\equiv 2$ and $s(.) \equiv s\in (0,1)$ constant, and of the $q$-Laplacian $-\Delta_q$, which corresponds to $q(.)\equiv q\in (1,+\infty)$ constant and $s(.) \equiv 1$.

System \eqref{mainProblem} can be cast as an extension to the fractional variable-order case of the second-order elliptic equation with variable growth conditions
\begin{equation}\label{04}
	\begin{cases}
		-\Delta u + \lambda Vu = \alpha \left\vert u\right\vert^{p(.)-2}u + \beta \left\vert u\right\vert^{k(.)-2}u & \text{ in }\Omega,
		\\
		u = 0 & \text{ on }\partial\Omega,
	\end{cases}
\end{equation}
which is obtained when considering $s(.)\equiv 1$ and $q(.)\equiv 2$. Equation \eqref{04} is a well-known model for electrorheological fluids \cite{Ruzicka3}, whose properties have been studied for instance in \cite{Alves,Mihailescu1,Mihailescu2}. 

The research on fractional Laplace operators and their applications is very attractive and extended. In the last decade, many authors from different fields of the pure and applied mathematics have considered PDE models involving the fractional Laplacian and addressed many relevant questions such as existence, uniqueness and regularity of solutions \cite{Autuori,Biccari3,Biccari6,Caffarelli1,Caffarelli,Nezza,Mingqi1,Mingqi2,Molica,Servadei1,Servadei2,Warma1,Xiang1,Xiang2,Xiang3,Xiang4}, spectral properties \cite{Farcaseanu1,Farcaseanu2,Lindgren}, or even more applied issues, for example control problems \cite{Biccari1,Biccari2,Biccari4,Biccari5,Warma2,Warma3,Warma4,Warma5} or the description of several phenomena arising in finance and quantum mechanics \cite{Applebaum,Caffarelli1,Laskin}.

On the other hand, results for the variable-order fractional Laplacian are limited and rare, and the literature on this operator is much less extended. We refer for instance to \cite{Bahrouni1,Bahrouni2,Kaufmann,Mingqi} for some relevant bibliography. In particular, in \cite{Mingqi}, the authors considered equation \eqref{mainProblem} with $q(.)\equiv2$ in a bounded domain of $\mathbb{R}^{n}$ and, under some suitable assumptions, they showed that the problem admits at least two different solutions for all $\lambda>0$. Moreover, they proved that these two solutions converge to two solutions of a limit model as $\lambda \rightarrow +\infty $, for which they also obtained the existence of infinitely many solutions.

The purpose of this paper is to extend the results of \cite{Mingqi} to the operator \eqref{fl}. In particular, we will show also in this case the existence of two distinct solutions for all $\lambda>0$, and that these two solutions converge to two of the infinitely many solutions of a limit model as $\lambda\to +\infty$.

Our results will be obtained by similar techniques as \cite{Mingqi}, based on Ekeland's variational principle and a mountain pass theorem, which we suitably adapted to cover the case of a variable $q(.)$. Let us stress that some of these techniques had already been employed in our previous contributions \cite{rahmoune1,rahmoune2,rahmoune3}, in the context of wave-type equations with variable-exponent nonlinearities.

This paper is composed by three sections in addition to the introduction. In Section \ref{sec:functionalSetting}, we recall the definitions of the variable-exponent Lebesgue and Sobolev spaces, and present some of their relevant properties. We also state there our main results. In Section \ref{sec:proof1}, we prove our first result showing the existence of a least two distinct nontrivial weak solutions for \eqref{mainProblem}. In Section \ref{sec:proof2}, we focus on the case $\lambda\to +\infty$ and prove the existence of infinitely many solutions for this limit problem. Finally, Section \ref{sec:open_pb} is devoted to some conclusions and open problems.

\section{Functional setting and main results}\label{sec:functionalSetting}

In this section, we describe the functional setting in which we shall work and state our main results.

Let us start by introducing the Lebesgue and Sobolev spaces with variable exponent. Here we refer mainly to \cite{Diening1,Fan,Fu,Kovacik}.

Throughout the rest of the paper we assume that $\Omega $ is a bounded open domain of $\mathbb{R}^{n}$, $n\geq 1$, with smooth boundary $\Gamma$. Moreover, in what follows, if not stated differently, we will always assume that $p:\overline{\Omega}\to (1,+\infty)$ is a measurable function and we will denote
\begin{align*}
	p^-:=\underset{x\in \Omega }{\text{ess}\inf}[\,p\left( x\right)] \quad\text{ and }\quad p^+:=\underset{x\in \Omega }{\text{ess}\sup}[\,p\left( x\right)]
\end{align*}
Let
\begin{align*}
	\varrho_{p(.),\Omega}(v):=\int_{\Omega}\left\vert v(x) \right\vert^{p(x)}\mathrm{d}x.
\end{align*}
We then define the variable-exponent space $L^{p(.)}(\Omega)$ as
\begin{displaymath}
	L^{p(.)}(\Omega )=\left\{v:\Omega \rightarrow \mathbb{R} \text{ measurable } \Big|\;\varrho_{p(.),\Omega}(v)<+\infty\right\},
\end{displaymath}
which is a Banach space equipped with the Luxemburg norm
\begin{equation*}
	\left\Vert u\right\Vert_{L^{p(.)}(\Omega)}:=\inf \left\{ \lambda >0 \;\Big|\; \int_{\Omega}\left\vert \frac{u(x)}{\lambda}\right\vert ^{p(x)}\mathrm{d}x\leq 1\right\}.
\end{equation*}

Variable-exponent Lebesgue spaces are similar to classical Lebesgue spaces in many aspects (see for instance \cite{Kovacik}). In particular,
it follows directly from the definition of the norm that
\begin{equation}\label{05}
	\min \left(\left\Vert u\right\Vert_{L^{p(.)}(\Omega)}^{p^-},\left\Vert u\right\Vert_{L^{p(.)}(\Omega)}^{p^+}\right) \leq \varrho_{p(.),\Omega}(u) \leq \max \left( \left\Vert u\right\Vert_{L^{p(.)}(\Omega)}^{p^-},\left\Vert u\right\Vert_{L^{p(.)}(\Omega)}^{p^+}\right).
\end{equation}
Moreover, we have the following generalized H\"{o}lder's inequalities.

\begin{theorem}[{\cite[Theorem 2.1]{Kovacik}}]\label{Holder1}
Let $p:\overline{\Omega}\rightarrow (1,+\infty)$ measurable and define the conjugate exponent
\begin{displaymath}
	p^{\prime}(x) = \frac{p(x)}{p(x)-1}, \quad (p^{\prime})^- = \frac{p^-}{p^- -1}
\end{displaymath}
so that we have
\begin{displaymath}
	\frac{1}{p(x)} + \frac{1}{p^{\prime}(x)} = 1.
\end{displaymath}
Then, for all functions $u\in L^{p(.)}(\Omega)$ and $v\in L^{p^{\prime}(.)}(\Omega)$, we have 	
\begin{equation*}
	\int_{\Omega}\left\vert u(x)v(x) \right\vert \mathrm{d}x\leq \left(\frac{1}{p^-}+\frac{1}{(p^\prime)^-}\right) \left\Vert u\right\Vert_{L^{p(.)}(\Omega)}\left\Vert v\right\Vert_{L^{p^\prime(.)}(\Omega)}.
\end{equation*}
\end{theorem}

\begin{theorem}[{\cite[Lemma 3.2.20]{Diening7}}]\label{Holder2}
Let $p,q,r:\overline{\Omega} \rightarrow (1,+\infty)$  be measurable functions such that
\begin{align*}
	\frac{1}{p(.)}=\frac{1}{r(.)} + \frac{1}{q(.)}.
\end{align*}
Then, for all functions $u\in L^{r(.)}(\Omega)$ and $v\in L^{q(.)}(\Omega)$, we have $uv\in L^{p(.)}(\Omega)$ with
\begin{equation*}
	\left\Vert uv\right\Vert_{L^{p(.)}(\Omega)}\leq \C\left\Vert u\right\Vert_{L^{r(.)}(\Omega)}\left\Vert v\right\Vert_{L^{q(.)}(\Omega)}.
\end{equation*}
\end{theorem}

Let us now introduce the variable order fractional Sobolev spaces. To this end, we shall make the following assumptions:

\medskip
\noindent\textbf{Hypothesis P:} $p:\overline{\Omega}\rightarrow (1,+\infty)$ is a measurable function satisfying:

\begin{subequations}
	\begin{align}
		&2<p^-\leq p(x) \leq p^+<+\infty \tag{\textbf{P1}}\label{hypP1}
		\\[10pt]
		&\displaystyle\left\vert p(x) -p(y) \right\vert \leq \frac{M}{\left\vert \log \left\vert x-y\right\vert \right\vert}\quad  \text{for all }x,y\text{ in }\Omega \text{ with }\left\vert x-y\right\vert <\frac{1}{2},\text{ }M>0 \tag{\textbf{P2}}\label{hypP2}
	\end{align}
\end{subequations}

\medskip
\noindent\textbf{Hypothesis Q:} $q:\overline{\Omega}\times \overline{\Omega}\rightarrow \mathbb{R}$ is a measurable function satisfying:
\begin{subequations}
	\begin{align}
		&q\text{ is symmetric, i.e., } q(x,y)=q(y,x) \text{ for all } (x,y)\in \overline{\Omega}\times \overline{\Omega} \tag{\textbf{Q1}}\label{hypQ1}
		\\[10pt]
		&1<q^-:=\min_{\overline{\Omega }\times \overline{\Omega}}\;q(x,y)\leq q(x,y)\leq \max_{\overline{\Omega}\times \overline{\Omega}}\;q(x,y)
		=: q^+<p^-<+\infty \tag{\textbf{Q2}}\label{hypQ2}
		\\[10pt]
		&q\big((x,y)-(z,z)\big)=q(x,y),\; \text{ for all } (x,y),(z,z)\in \Omega\times\Omega \tag{\textbf{Q3}}\label{hypQ3}
	\end{align}
\end{subequations}

\medskip
\noindent\textbf{Hypothesis S:} $s:\mathbb{R}^n\times\mathbb{R}^n\rightarrow (0,1)$ is a measurable function satisfying:
\begin{subequations}
	\begin{align}
		&s\text{ is symmetric, i.e., } s(x,y)=s(y,x) \text{ for all } (x,y)\in \mathbb{R}^n\times \mathbb{R}^n \tag{\textbf{S1}}\label{hypS1}
		\\[10pt]
		&0<s^-:=\min_{\mathbb{R}^{2n}}\;s(x,y)\leq s(x,y) \leq \max_{\mathbb{R}^{2n}}\;s(x,y)=:s^+<1 \tag{\textbf{S2}}\label{hypS2}
	\end{align}
\end{subequations}

We then define the generalized fractional Sobolev space with variable exponents via the Gagliardo approach as follows
\begin{equation*}
	\mathcal{H}^{s(.)}(\Omega) = H^{p(.),q(.),s(.)}(\Omega) =\left\{u\in L^{p(.)}(\Omega)\;\Big|\; [u]_{q(.),s(.),\Omega }<+\infty \right\},
\end{equation*}
where
\begin{equation*}
	[u]_{q(.),s(.),\Omega}=\inf \left \{\lambda>0 \;\bigg|\; \int_{\Omega}\int_{\Omega}\left\vert \frac{u(x)-u(y)}{\lambda}\right\vert
	^{q(x,y)}\!\!\frac{\mathrm{d}x\mathrm{d}y}{\left\vert x-y\right\vert ^{n+q(x,y)s(x,y)}}<1\right\}
\end{equation*}
is the corresponding variable exponent Gagliardo seminorm. Then, $\mathcal{H}^{s(.)}(\Omega)$ equipped with the norm
\begin{equation*}
	\left\Vert u\right\Vert_{\mathcal{H}^{s(.)}(\Omega)}=\left\Vert u\right\Vert_{L^{p(.)}(\Omega)} + [u]_{q(.),s(.),\Omega}
\end{equation*}
is a Banach space.

Define now $\mathcal{H}_0^{s(.)}(\Omega)=H_0^{p(.),q(.),s(.)}(\Omega)$ as the linear space of Lebesgue measurable functions $u:\mathbb{R}^n\rightarrow\mathbb{R}$ such that $u\in H^{p(.),q(.),s(.)}(\Omega)$ with $u=0$ in $\mathbb{R}^n\backslash\Omega$. Then, $\mathcal{H}_0^{s(.)}(\Omega)$ is a Banach space with the norm endowed by $\mathcal{H}^{s(.)}(\Omega)$. Moreover, we have the following result.

\begin{proposition}\label{Thes2}
Let $p(x)$, $q(x,y)$ and $s(x,y)$ be continuous variable exponents and define
\begin{align*}
	s^-:=\min_{\overline{\Omega }\times \overline{\Omega}} s(x,y).
\end{align*}
Assume that
\begin{subequations}
	\begin{flalign}
		n>s^-q(x,y) &\quad\text{ for  all } (x,y) \in\overline{\Omega}\times \overline{\Omega} \label{con01}
		\\[5pt]
		p(x)>q(x,x) &\quad\text{ for all } x\in \overline{\Omega} \label{con02}
	\end{flalign}	
\end{subequations}
Let \eqref{hypP1}, \eqref{hypP2}, \eqref{hypQ1}, \eqref{hypQ2} and \eqref{hypQ3} be satisfied. Assume that $r:\overline{\Omega}\rightarrow (1,+\infty)$ is a continuous function such that
\begin{equation}\label{con1}
	q^{\ast }(x):=\frac{nq(x,x)}{n-s^-q(x,x)}>r(x), \quad \text{ for all } x\in \overline{\Omega}.
\end{equation}
Then
\begin{itemize}
	\item[1.] There exists a constant $\C=\C(n,p,q,r,s,\Omega)$ such that for every $v\in \mathcal{H}_0^{s(.)}(\Omega)$ we have
	\begin{equation*}
		\left\Vert v\right\Vert_{L^{r(.)}(\Omega)}\leq \C\left\Vert v\right\Vert_{\mathcal{H}_0^{s(.)}(\Omega)},
	\end{equation*}
	i.e. $\mathcal{H}_0^{s(.)}(\Omega)$ can be continuously embedded into $L^{r(.)}(\Omega)$ for any $r\in (1,q^{\ast })$.
	\item[2.] The embedding $\mathcal{H}_0^{s(.)}(\Omega)\hookrightarrow L^{r(.)}(\Omega)$ is compact.
	\item[3.] When one considers functions $u\in\mathcal{H}_0^{s(.)}(\Omega)$ that are compactly supported inside $\Omega$, the embeddings $\mathcal{H}_0^{s^+}(\Omega) \hookrightarrow \mathcal{H}_0^{s(.)}(\Omega) \hookrightarrow \mathcal{H}_0^{s^-}(\Omega)$ are continuous.
\end{itemize}
\end{proposition}

\begin{proof}
The proof will be organized in three steps, one for each different result we stated. In particular, in Step 1, we will show that $\mathcal{H}_0^{s(.)}(\Omega)$ can be continuously embedded into $L^{r(.)}(\Omega)$ for any $r\in (1,q^{\ast })$. In Step 2, we will show that this embedding is compact. Finally, in Step 3, we will show that when one considers functions $u\in\mathcal{H}_0^{s(.)}(\Omega)$ that are compactly supported inside $\Omega$, the embeddings $\mathcal{H}_0^{s^+}(\Omega) \hookrightarrow \mathcal{H}_0^{s(.)}(\Omega) \hookrightarrow \mathcal{H}_0^{s^-}(\Omega)$ are continuous.
	
\textbf{Step 1.} First of all, by definition of $s^{-}$ and using \eqref{con01} and \eqref{con1} we have that there exists a constant $k_1>0$ such that
\begin{align*}
	\frac{nq(x,x)}{n-s(.)q(x,x)}-r(x) \geq k_1>0 \quad \text{ for all } x\in\overline{\Omega}.
\end{align*}

Moreover, by \eqref{con02} we have the existence of a second positive constant $k_2$ such that
\begin{align*}
	p(x)-q(x,x) \geq k_2>0.
\end{align*}

Thus, there exists a constant $\varepsilon>0$ and $K$ numbers of disjoint hypercubes $B_i$ such that $\Omega =\cup_{i=1}^KB_i$ and $\text{diam}(B_i)<\varepsilon$, that verify
\begin{equation}\label{8}
	\begin{array}{l}
		\displaystyle\frac{nq(z,y)}{n-s(z,y) q(z,y)} -r(x) \geq \frac{k_1}{2},
		\\[10pt]
		\displaystyle p(x)\geq q(z,y)+\frac{k_2}{2},
	\end{array}
\end{equation}
for every $x\in B_i$ and $(z,y)\in B_i\times B_i$. Let
\begin{align*}
	s_i= \inf_{B_i\times B_i}\;s(z,y),\quad q_i=\inf_{B_i\times B_i}(q(z,y)-\delta) \quad \text{ and } \quad q_i^{\ast}=\frac{nq_i}{n-s_iq_i}.
\end{align*}

From \eqref{8} and the continuity of the associated exponents, we can pick $\delta=\delta(k_1)$, with $q^- -1>\delta >0$, such that
\begin{displaymath}
	\begin{array}{ll}
		\displaystyle\frac{nq_i}{n-s_iq_i}\geq r(x) +\frac{k_1}{2} &\text{ for all } x\in B_i \text{ and } n>s_iq_i,
		\\[10pt]
		\displaystyle p(x)\geq q_i+\frac{k_2}{2} & \text{ for all }x\in B_i.
	\end{array}
\end{displaymath}

Therefore we can employ the Sobolev embedding theorem for constant exponents (see \cite[Theorem 5.4]{Adams}) to get the existence of a suitable constant $\C=\C(n,q_i,s_i,\varepsilon,B_i)$ such that
\begin{equation}\label{11}
	\left\Vert u\right\Vert_{L^{q_i^{\ast }}(B_i)}\leq \C\left(\left\Vert u\right\Vert_{L^{q_i}(B_i)} + [u]_{q_i,s_i,B_i}\right).
\end{equation}

Let us now suppose that there exist three positive constants $c_1$, $c_2$ and $c_3$ such that
\begin{subequations}
	\begin{align}
		&\sum_{i=0}^n\left\Vert u\right\Vert_{L^{q_i^{\ast }}(B_i)}\geq c_1\left\Vert u\right\Vert_{L^{r(.)}(\Omega)} \label{ineq1}
		\\
		&\left\Vert u\right\Vert_{L^{p(.)}(\Omega)}\geq c_2\sum_{i=0}^n\left\Vert u\right\Vert_{L^{q_i}(B_i)} \label{ineq2}
		\\
		&[u]_{q(.),s(.),\Omega}\geq c_3\sum_{i=0}^n[u]_{q_i,s_i,B_i} \label{ineq3}
	\end{align}
\end{subequations}
Then, from \eqref{11} and \eqref{ineq1}-\eqref{ineq3} we can conclude that
\begin{align*}
	\left\Vert u\right\Vert_{L^{r(.)}(\Omega)}&\leq c_1\sum_{i=0}^n\left\Vert u\right\Vert_{L^{q_i^{\ast }}(B_i)} \leq \C\sum_{i=0}^n\left(\left\Vert u\right\Vert_{L^{q_i}(B_i)} + [u]_{q_i,s_i,B_i}\right)
	\\
	&\leq \C\left(\left\Vert u\right\Vert_{L^{p(.)}(\Omega)} + [u]_{q(.),s(.),\Omega}\right) = \C\left\Vert u\right\Vert_{\mathcal{H}_0^{s(.)}(\Omega)},
\end{align*}
as we wanted to show.
	
Hence, we only have to prove that \eqref{ineq1}-\eqref{ineq3} hold. Let us start with \eqref{ineq1}. We have
\begin{equation*}
	\left\vert u(x) \right\vert = \sum_{i=0}^n\left\vert u(x) \right\vert \chi_{B_i},
\end{equation*}
which clearly implies that
\begin{equation}\label{12}
	\left\Vert u\right\Vert_{L^{r(.)}(\Omega)} \leq \sum_{i=0}^n\left\Vert u\right\Vert_{L^{r(.)}(B_i)}
\end{equation}

Moreover, notice that for each $i$, $q_i^{\ast }>r(x)$ if $x\in B_i$. Then we can choose $a_i(x)$ such that
\begin{equation*}
	\frac{1}{r(x)}=\frac{1}{q_i^{\ast}} + \frac{1}{a_i(x)},
\end{equation*}
and by Theorem \ref{Holder2} we have
\begin{equation*}
	\left\Vert u\right\Vert_{L^{r(.)}(B_i)} \leq \C\left\Vert 1\right\Vert_{L^{a_i(.)}(B_i)}\left\Vert u\right\Vert_{L^{q_i^{\ast }}(B_i)} = \C\left\Vert u\right\Vert_{L^{q_i^{\ast}}(B_i)}.
\end{equation*}

Thus, \eqref{ineq1} immediately follows from \eqref{12}. Moreover, in a similar way and using that $p(x)>q_i$ for $x\in B_i$, we easily obtain also \eqref{ineq2}. Finally, to prove \eqref{ineq3} let us fix
\begin{equation*}
	U(x,y) = \frac{\left\vert u(x) - u(y) \right\vert}{\left\vert x-y\right\vert^{s_i}},
\end{equation*}
and remark that
\begin{align}\label{09}
	[u]_{q_i,s_i,B_i} &= \iint_{B_i\times B_i} \frac{\left\vert u(x)-u(y)\right\vert^{q_i}}{\left\vert x-y\right\vert^{n+q_i s_i}}\,\mathrm{d}x\mathrm{d}y =\iint_{B_i\times B_i}\left(\frac{\left\vert u(x) - u(y) \right\vert}{\left\vert x-y\right\vert^{s_i}}\right)^{q_i} \frac{\mathrm{d}x\mathrm{d}y}{\left\vert x-y\right\vert^n}  \notag
	\\
	&=\left\Vert U\right\Vert_{L^{q_i}(\mu,B_i\times B_i)}^{q_i}\leq \C\left\Vert 1\right\Vert_{L^{b_i(.,.)}(\mu,B_i\times B_i)}^{q_i}\left\Vert U\right\Vert_{L^{q(.,.)}(\mu,B_i\times B_i)}^{q_i}  \notag
	\\
	&=\C\left\Vert U\right\Vert_{L^{q(.,.)}(\mu,B_i\times B_i)}^{q_i},
\end{align}
where Theorem \ref{Holder2} is used with $b_i(x,y)$ such that
\begin{equation*}
	1=\frac{q_i}{q(x,y)}+\frac{q_i}{b_i(x,y)},
\end{equation*}
but considering the measure in $B_i\times B_i$ given by
\begin{equation*}
	\mathrm{d}\mu (x,y) =\frac{\mathrm{d}x\mathrm{d}y}{\left\vert x-y\right\vert^{\frac{n}{q_i}}}.
\end{equation*}
Now we aim to show that
\begin{equation}\label{13}
	\left\Vert U\right\Vert_{L^{q(.,.)}(\mu,B_i\times B_i)}\leq \C[u]_{q(.),s(.),B_i},
\end{equation}
for every $i$. If this is valid, then we directly obtain \eqref{ineq3} from \eqref{09}. Let $\lambda >0$ be such that
\begin{equation*}
	\iint_{B_i\times B_i}\left\vert \frac{u(x) - u(y)}{\lambda}\right\vert^{q(x,y)}\frac{\mathrm{d}x\mathrm{d}y}{\left\vert x-y\right\vert^{n+q(x,y) s(x,y)}}<1.
\end{equation*}
Then
\begin{align*}
	\iint_{B_i\times B_i} &\left(\frac{\left\vert u(x) - u(y) \right\vert}{\lambda \left\vert x-y\right\vert^{s(x,y)}}\right)^{q(x,y)}\frac{\mathrm{d}x\mathrm{d}y}{\left\vert x-y\right\vert^n}
	\\
	&= \iint_{B_i\times B_i}\left\vert \frac{u(x) - u(y)}{\lambda}\right\vert^{q(x,y)}\frac{\mathrm{d}x\mathrm{d}y}{\left\vert x-y\right\vert^{n+q(x,y)s(x,y)}}<1.
\end{align*}
Therefore
\begin{equation*}
	\left\Vert U\right\Vert_{L^{q(.,.)}(\mu,B_i\times B_i)}\leq \lambda,
\end{equation*}
and we finally obtain the inequality \eqref{13}.
	
\textbf{Step 2.} Let us consider a sequence $\{u_j\}_j\subset \mathcal{H}_0^{s(.)}(\Omega)$ such that $u_j\rightarrow u$ in $\mathcal{H}_0^{s(.)}(\Omega)$ as $j\to +\infty$, and denote $v_j:=u_j-u$. Hence $v_j\rightarrow 0$ in $\mathcal{H}_0^{s(.)}(\Omega)$, which implies that $[v_j]_{q(.),s(.),\Omega}$ is uniformly bounded.

Extend the functions $v_j$ by zero outside of $\Omega$ and, with some abuse of notation, denote this extension $v_j$. We have to show that $v_j\rightarrow 0$ in $L^{r(.)}(\Omega)$. To this end, let $\psi_{\varepsilon}$ be a standard mollifier for all $\varepsilon >0$. We then have
\begin{align*}
	v_j=\left(v_j-\psi_{\varepsilon}\ast v_j\right)+\psi_{\varepsilon }\ast v_j
\end{align*}
and from \eqref{ineq11} we get
\begin{align}\label{ineq11}
	\left\Vert v_j\right\Vert_{L^{r(.)}(\Omega)} &\leq \left\Vert v_j-v_j\ast \psi_{\varepsilon}\right\Vert_{L^{r(.)}(\Omega)} + \left\Vert v_j\ast \psi_{\varepsilon}\right\Vert_{L^{r(.)}(\Omega)} 
	\\
	&\leq c\varepsilon [v_j]_{q(.),s(.),\Omega} + \left\Vert v_j\ast \psi_{\varepsilon}\right\Vert_{L^{r(.)}(\Omega)}. \notag
\end{align}
Since $v_j\rightarrow 0$ and $\varepsilon >0$ is fixed, we obtain that as $j\rightarrow +\infty$
\begin{equation*}
	v_j\ast \psi_{\varepsilon}(x) = \int_{\mathbb{R}^n}\psi_{\varepsilon}(x-y)v_j(y)dy\rightarrow 0
\end{equation*}

Let $\Omega_{\varepsilon}:=\left\{x\in \mathbb{R}^n \,|\, \func{dist}(x,\Omega)\leq \varepsilon \right\}$. Thus $v_j\ast \psi_{\varepsilon}(x)=0$ for all $x\in \mathbb{R}^n\backslash \Omega_{\varepsilon}$. By the H\"older's inequality given in Theorem \ref{Holder1}, we then get that for all $x\in \Omega_{\varepsilon}$
\begin{equation*}
	\left\vert v_j\ast \psi_{\varepsilon}(x)\right\vert =\left\vert \int_{\mathbb{R}^n}\psi_{\varepsilon}(x-y)v_j(y)dy\right\vert \leq c\left\Vert v_j\right\Vert_{L^{r(.)}(\Omega)}\left\Vert \psi_{\varepsilon}(x-.)\right\Vert_{L^{r^\prime(.)}(\Omega)}.
\end{equation*}
As $\psi \in C_0^{\infty}(\mathbb{R}^n)$, we have $|\psi|\leq c$ and thus $\left\vert \psi _{\varepsilon }\right\vert \leq
c\varepsilon ^{-n}$. This gives
\begin{align*}
	\left\Vert \psi_{\varepsilon}(x-.)\right\Vert_{L^{r^\prime(.)}(\Omega)}\leq c\varepsilon^{-n}\left\Vert \chi_{\Omega_{\varepsilon }}\right\Vert_{L^{r^\prime(.)}(\Omega)}\leq c(\varepsilon,r),
\end{align*}
independently of the choice of $x\in \mathbb{R}^n$ and $j$. Using the uniform boundedness of $v_j$ in $L^{r(.)}$, we then have
\begin{equation*}
	\left\vert v_j\ast \psi_{\varepsilon}(x)\right\vert \leq c(\varepsilon,r)\chi_{\Omega_{\varepsilon}}(x) \quad\text{ for all }x\in \mathbb{R}^n.
\end{equation*}

Since $\chi_{\Omega_{\varepsilon }}\in L^{r(.)}(\mathbb{R}^n)$ and $v_j\ast \psi_{\varepsilon}(x)\rightarrow 0$ a. e., we obtain by the dominated convergence theorem that $v_j\ast \psi_{\varepsilon}\rightarrow 0$ in $L^{r(.)}(\mathbb{R}^n)$ as $j\rightarrow +\infty$. Hence, from \eqref{ineq11} it follows that
\begin{equation*}
	\limsup_{j\rightarrow +\infty}\left\Vert v_j\right\Vert_{L^{r(.)}(\Omega)}\leq c\varepsilon \limsup_{j\rightarrow +\infty}[v_j]_{q(.),s(.),\Omega}.
\end{equation*}

Since $\varepsilon >0$ is arbitrary and $[v_j]_{q(.),s(.),\Omega}$ is uniformly bounded, this implies that $v_j\rightarrow 0$ in $L^{r(.)}(\mathbb{R}^n)$ and thus $u_j\rightarrow u$ in $L^{r(.)}(\Omega)$, which yields the compactness of the embedding.

\textbf{Step 3.} When we consider functions that are compactly supported inside $\Omega$, we can get rid of the term $\left\Vert u\right\Vert_{L^{p(.)}(\Omega)}$ and it holds that
\begin{equation*}
	\left\Vert u\right\Vert_{L^{p(.)}(\Omega)}\leq\C [u]_{q(.),s(.),\Omega}.
\end{equation*}

Let $u\in \mathcal{H}_0^{s(.)}(\Omega )$. From \eqref{ineq3}, and because $s^{-}\leq s_{i}$ in $B_{i}$ for all $i$, since in the case of constant exponent $s_{i}$ the Sobolev embedding for subcritical exponents is continuous, we have 
\begin{equation*}
   \sum_{i=0}^{n}[u]_{q_i,s^-,B_i}\leq c\sum_{i=0}^n[u]_{q_i,s_i,B_i}\leq c[u]_{q(.),s(.),\Omega}<+\infty.
\end{equation*}

Using this for every $i$, it holds that $[u]_{q(.),s^-,\Omega}\leq c[u]_{q(.),s(.),\Omega}$, which gives $u\in \mathcal{H}_0^{s^-}(\Omega)$, and then $\mathcal{H}_0^{s(.)}(\Omega)\hookrightarrow \mathcal{H}_0^{s^-}(\Omega)$.

It only remains to show that $\mathcal{H}_0^{s^+}(\Omega)\hookrightarrow \mathcal{H}_0^{s(.)}(\Omega)$. To this end, let $u\in \mathcal{H}_0^{s^+}(\Omega)$, and let $\lambda >0$ be such that
\begin{equation*}
   \iint_{B_i\times B_i}\left\vert \frac{u(x)-u(y)}{\lambda}\right\vert^{q(x,y)}\frac{\mathrm{d}x\mathrm{d}y}{\left\vert x-y\right\vert ^{n+q(x,y)s^+}}<1.
\end{equation*}

Then, using that $diam(B_i)<\varepsilon <1$, we get $\left\vert x-y\right\vert <1$ for every $(x,y)\in B_i\times B_i$ and hence
\begin{align*}
	\iint_{B_i\times B_i} &\left\vert \frac{u(x)-u(y)}{\lambda}\right\vert^{q(x,y)}\frac{\mathrm{d}x\mathrm{d}y}{\left\vert x-y\right\vert^{n+q(x,y)s(x,y)}} 
	\\
	&=\iint_{B_i\times B_i}\left\vert \frac{u(x)-u(y)}{\lambda}\right\vert^{q(x,y)}\frac{\left\vert x-y\right\vert^{n+q(x,y)s^+}}{\left\vert x-y\right\vert^{n+q(x,y)s(x,y)}}\frac{\mathrm{d}x\mathrm{d}y}{\left\vert x-y\right\vert^{n+q(x,y)s^{+}}}
	\\
	&\leq \iint_{B_i\times B_i}\left\vert \frac{u(x)-u(y)}{\lambda}\right\vert^{q(x,y)}\frac{\mathrm{d}x\mathrm{d}y}{\left\vert x-y\right\vert^{n+q(x,y)s^+}}<1
\end{align*}
Therefore $u\in \mathcal{H}_0^{s(.)}(\Omega)$, that is, $\mathcal{H}_0^{s^+}(\Omega)\hookrightarrow \mathcal{H}_0^{s(.)}(\Omega)$.
\end{proof}

\begin{remark}
Our result is sharp in the following sense: if
\begin{equation*}
    q^{\ast }(x_{0})=\frac{nq(x_{0},x_{0})}{n-s(x_{0},x_{0})q(x_{0},x_{0})}<r(x_{0})
\end{equation*}%
for some $x_{0}\in \Omega $, then the embedding of $\mathcal{H}_{0}^{s(.)}(\Omega )$ in $L^{r(.)}(\Omega )$ cannot hold for every $p(.)$. As a matter of fact, from the continuity conditions we imposed on $q,$ $r$ and $s$, there exists a small ball $B_{\delta }(x_{0})$ such that
\begin{equation*}
  \max_{\overline{B}_{\delta }(x_{0})\times \overline{B}_{\delta }(x_{0})}\;\frac{nq(x,y)}{n-s(x,y)q(x,y)}<\min_{\overline{B}_{\delta }(x_{0})}\;r(x).
\end{equation*}%
Now, fix $p(x)<\min_{\overline{B}_{\delta }(x_{0})}\;r(x)$ and notice that, for $p(x)\geq r(x)$, we have that $\mathcal{H}_{0}^{s(.)}(\Omega )$ is embedded in $L^{r(.)}(\Omega )$. Hence, with the same arguments that hold for the constant exponent case, one can create a sequence $\{u_j\}_j$ supported inside $B_{\delta }(x_{0})$ such that $\left\Vert u_j\right\Vert _{\mathcal{H}_{0}^{s(.)}(\Omega)}\leq \mathcal{C}$ and $\left\Vert u_j\right\Vert_{L^{r(.)}(B_{\delta}(x_0))}\rightarrow +\infty $. In fact, it is enough to consider a smooth, compactly supported function $g$, and pick $u_j=j^{b}g(jx)$ with $b$ satisfying $bq(x,y)-n+s(x,y)q(x,y)\leq 0$ and $br(x)-n>0$ for $x,y\in \overline{B}_{\delta }(x_0)$.
\end{remark}

In what follows, we will always denote by $\C_r$ the constant of the Sobolev embedding $\mathcal{H}_0^{s(.)}(\Omega) \hookrightarrow L^{r(.)}(\Omega)$. Then, by applying \eqref{05} and Proposition \ref{Thes2}, for all
\begin{align*}
	p(x)\in \left(1,\frac{nq(x,x)}{n-s(x,x) q(x,x)}\right)
\end{align*}
we obtain
\begin{align}\label{06}
	\int_{\Omega }\left\vert u(x)\right\vert ^{p(x)}\mathrm{d}x& \leq \max\left( \left\Vert u\right\Vert _{L^{p(.)}(\Omega)}^{p^{-}},\left\Vert u\right\Vert_{L^{p(.)}(\Omega)}^{p^{+}}\right)
	\\
	& \leq \max \left(\C_p^{p^{-}}[u]_{q(.),s(.),\Omega}^{p^{-}},\C_p^{p^{+}}[u]_{q(.),s(.),\Omega}^{p^{+}}\right). \notag
\end{align}

We are now ready to provide our notion of solution to \eqref{mainProblem}. To this end, let us introduce the Banach space
\begin{equation*}
	E_{\lambda}=\left\{u\in L^{p(x)}(\Omega)\;\Big|\; [u]_{q(.),s(.),\Omega} + \lambda \int_{\Omega}V(x)\left\vert u(x) \right\vert^2\mathrm{d}x<+\infty \right\}
\end{equation*}
equipped with the norm $[u]_{q(.),s(.),\Omega}+\left\Vert u\right\Vert_{p(.)}$. Let $E_{\lambda,0}$ denote the closure of $C_0^{\infty}(\Omega)$ in $E_{\lambda}$. Then $E_{\lambda,0}$ is a Banach space with the norm
\begin{align*}
	\left\Vert u\right\Vert_{\lambda }:=[u]_{q(.),s(.),\Omega}.
\end{align*}

This space $E_{\lambda,0}$ is the appropriate functional setting in which we can define our notion of solution to \eqref{mainProblem}. In particular, we have
\begin{definition}\label{solDef}
A function $u\in E_{\lambda,0}$ is called a (weak) solution of problem \eqref{mainProblem} if
\begin{align*}
	\iint_{\mathbb{R}^{2n}} & \frac{\left\vert u(x) - u(y)	\right\vert^{q(x,y)-2}(u(x) - u(y))(v(x) - v(y))}{\left\vert x-y\right\vert^{n+q(x,y)s(x,y)}}\,\mathrm{d}x\mathrm{d}y
	\\
	&+ \lambda \int_{\Omega}V(x)u(x) v(x)\,\mathrm{d}x =\int_{\Omega}\left(\alpha \left\vert u\right\vert^{p(x)-2}uv(x) + \beta \left\vert u\right\vert^{k(x)-2}u(x) v(x)\right) \mathrm{d}x,
	\end{align*}
for any $v\in E_{\lambda,0}$.
\end{definition}

Notice that, given the variational nature of Definition \ref{solDef}, the solution of \eqref{mainProblem} can be characterized in terms of the critical points of a suitable functional. In particular, let us define
\begin{align*}
	I_{\lambda}(u) :=&\, \iint_{\mathbb{R}^{2n}}\frac{1}{q(x,y)}\frac{\left\vert u(x) - u(y) \right\vert^{q(x,y)}}{\left\vert x-y\right\vert^{n+q(x,y)s(x,y)}}\,\mathrm{d}x\mathrm{d}y + \frac{\lambda}{2}\int_{\Omega}V(x) \left\vert u(x) \right\vert^{2}\,\mathrm{d}x
	\\
	&-\int_{\Omega}\left(\frac{\alpha}{p(x)}\left\vert u\right\vert^{p(x)} + \frac{\beta}{k(x)}\left\vert u\right\vert^{k(x)}\right) \mathrm{d}x, \quad\text{ for all } u\in E_{\lambda,0}
\end{align*}
and $L_1:E_{\lambda,0}\rightarrow E_{\lambda,0}^{\ast}$ such that
\begin{align*}
	\langle L_1(u),v\rangle_{\lambda} =& \iint_{\mathbb{R}^{2n}}\frac{\left\vert u(x) -u(y) \right\vert^{q(x,y) -2}(u(x) - u(y))(v(x) - v(y))}{\left\vert x-y\right\vert^{n+q(x,y)s(x,y)}}\,\mathrm{d}x\mathrm{d}y
	\\
	&+\lambda \int_{\Omega}V(x) u(x) v(x)\, \mathrm{d}x, \quad \text{ for all }u,v\in E_{\lambda,0},
\end{align*}
where we denoted with $\langle \cdot,\cdot\rangle_{\lambda}$ the duality pair between $E_{\lambda,0}$ and $E_{\lambda,0}^{\ast}$. Then, one can verify that $I_{\lambda}$ is well-defined, of class $C^1$ in $E_{\lambda,0}$ and, for all $u,v\in E_{\lambda,0}$, we have
\begin{equation*}
	\langle I_{\lambda }^{\prime}(u),v\rangle_{\lambda} = \langle L_1(u),v\rangle_{\lambda} - \int_{\Omega}\left(\alpha\left\vert u\right\vert^{p(x)-2}u + \beta \left\vert u\right\vert^{k(x)-2}u\right) v\,\mathrm{d}x,
\end{equation*}

Hence if $u\in E_{\lambda,0}$ is such that $\langle I_{\lambda}^{\prime}(u),v\rangle_{\lambda} = 0$ for all $v\in E_{\lambda,0}$, then $u$ is a solution of \eqref{mainProblem}.

We can now introduce the main results of this paper. To this end, we shall first make the following assumptions on the potential $V$ and the function $k$.

\noindent\textbf{Hypothesis V:} $V:\Omega\rightarrow [0,+\infty)$ is a continuous function satisfying:
\begin{subequations}
	\begin{align}
		&J=int(V^{-1}(0))\subset\Omega \text{ is a nonempty bounded domain and } \widetilde{J}=V^{-1}(0) \tag{\textbf{V1}}\label{hypV1}
		\\[10pt]
		&\text{ there exists a nonempty open domain } \Omega_0\subset J \tag{\textbf{V2}}\label{hypV2}
		\\
		&\text{ such that } V(x)\equiv 0 \text{ for all } x\in \overline{\Omega}_0 \notag
	\end{align}
\end{subequations}

\noindent\textbf{Hypothesis K:} $k:\Omega\rightarrow \mathbb{R}$ is a continuous function satisfying:
\begin{subequations}
	\begin{align}
		1<k^-\leq k(x)\leq k^+<2 \quad \text{ for all }x\in \overline{\Omega } \tag{\textbf{K1}}\label{hypK1}
	\end{align}
\end{subequations}
Moreover, we shall consider the following variant of \eqref{hypP1}
\begin{align}\label{hypPbis}
	2<p^-\leq p(x)\leq p^+<\frac{nq(x,x)}{n-s(x,x)q(x,x)} \quad \text{ for all }x\in \overline{\Omega } \tag{\textbf{P1a}}
\end{align}
Finally, we assume that the parameters $\alpha$ and $\beta $ verify
\begin{align}\label{hypAlpha}
	\alpha \leq \frac{D(2-k^+)}{A(p^+ -k^+)}, \quad \beta \leq \frac{D(p^+ - 2)}{B(p^+ - k^+)}
\end{align}
with
\begin{align}\label{constants}
	A  =\frac{\max \left(\C_p^{p^-},\C_p^{p^+}\right)}{p^-}, \quad B=\frac{\max \left(\C_k^{k^-},\C_k^{k^+}\right)}{k^-}, \quad D=\min	\left(\frac{1}{q^+},\frac{1}{2}\right)
\end{align}
where, we recall, $\C_p$ and $\C_k$ are the constants of the Sobolev embeddings
\begin{align*}
	\mathcal{H}_0^{s(.)}(\Omega) \hookrightarrow L^{p(.)}(\Omega) \quad\text{ and }\quad \mathcal{H}_0^{s(.)}(\Omega) \hookrightarrow L^{k(.)}(\Omega),
\end{align*}
respectively.

The first main contribution of the present paper is the non-unicity of solutions for \eqref{mainProblem}. In particular, we have:

\begin{theorem}\label{The1}
Assume that \eqref{hypPbis}, \eqref{hypP2}, \eqref{hypS1}, \eqref{hypS2}, \eqref{hypV1}, \eqref{hypV2}, \eqref{hypK1} and \eqref{hypAlpha} hold. Let $n>q^+s^+$. Then \eqref{mainProblem} admits at least two distinct solutions for all $\lambda >0$.
\end{theorem}

Furthermore, in the following results, we investigate the concentration of solutions obtained by Theorem \ref{The1}.

\begin{theorem}\label{The2}
Let $u_{\lambda}^{1}$ and $u_{\lambda}^{2}$ be two solutions obtained in Theorem \ref{The1} and $\Omega_0$ as in \eqref{hypV2}. Then $u_{\lambda}^{1}\rightarrow u^{1}$ and $u_{\lambda}^{2}\rightarrow u^{2}$ in $\mathcal{H}_0^{s(.)}(\Omega)$ as $\lambda \rightarrow +\infty $, where $u^{1}\neq u^{2}$ are two nontrivial solutions of the following problem
\begin{equation}\label{23}
	\begin{cases}
		(-\Delta)_{q(.)}^{s(.)}u =\alpha \left\vert u\right\vert^{p(.) -1}u + \beta\left\vert u\right\vert^{k(.) -1}u & \text{ in }\Omega_0,
		\\[7pt]
		u =0 & \text{ in }\mathbb{R}^{n}\backslash \Omega_0.
	\end{cases}
\end{equation}
\end{theorem}

\begin{theorem}\label{The3}
Assume that \eqref{hypPbis}, \eqref{hypS1}, \eqref{hypS2}, \eqref{hypK1} and \eqref{hypAlpha} hold. Let $n>2s^+$. Then problem \eqref{23} has infinitely many solutions.
\end{theorem}

\section{Proof of Theorem \ref{The1}}\label{sec:proof1}

We give here the proof of our first main result Theorem \ref{The1}. To this end, we first need some preparation.

Let us start by introducing the notion of $(PS)_c$ sequence and stating the so-called $(PS)_c$ condition.

\begin{definition}
For any $c\in\mathbb{R}$, a sequence $\{u_j\}_j\subset E_{\lambda}$ is called a $(PS)_c$ sequence if $I_{\lambda}(u_j)\rightarrow c$ and $I_{\lambda}^{\prime}(u_j)\rightarrow 0$ as $j\rightarrow +\infty$.
\end{definition}

\begin{definition}
We say that $I_{\lambda}$ provides the $(PS)_c$ condition in $E_{\lambda}$ at the level $c\in \mathbb{R}$ if each $(PS)_c$ sequence $\{u_j\}_j\subset E_{\lambda}$ possess a strongly convergent subsequence in $E_{\lambda}$.
\end{definition}

Moreover, in the sequel, we shall make use of the following standard mountain pass theorem (see for example \cite{Ambrosetti}).

\begin{theorem}\label{Thes01}
Let $E$ be a real Banach space and $J\in C^1(E,\mathbb{R})$ with $J(0)=0$. Suppose that
\begin{itemize}
	\item[(\textbf{i})] There exist $\rho,\delta >0$ such that $J(u)\geq\delta$ for all $u\in E$ with $\left\Vert u\right\Vert_E=\rho$.

	\item[(\textbf{ii})] There exists $e\in E$ satisfying $\left\Vert e\right\Vert_E>\rho $ such that $J(e)<0$.
\end{itemize}
Define $\Gamma =\{\gamma \in C^1([0,1];E)\;|\;\gamma (0)=1 \text{ and }\gamma (1)=e\}$. Then
\begin{equation*}
	c=\inf_{\gamma \in \Gamma}\max_{0\leq \sigma \leq 1}J(\gamma(\sigma))\geq \delta
\end{equation*}
and there exists a $(PS)_{c}$ sequence $\{u_j\}_j\subset E$.
\end{theorem}

Theorem \ref{Thes01} applied to the functional $I_\lambda$ will be the starting point to prove our main result. To this end, we first need to check that $I_{\lambda}$ possesses the mountain pass geometry (\textbf{i}) and (\textbf{ii}). This is ensured by the following lemma.

\begin{lemma}\label{Lemma1}
Assume that the assumptions \eqref{hypPbis}, \eqref{hypS2}, \eqref{hypV1}-\eqref{hypV2}, \eqref{hypK1} and \eqref{hypAlpha} are satisfied. Then
\begin{itemize}
	\item[1.] For all $\lambda>0$, there exist $\rho>0$ and $\delta>0$ such that
	\begin{equation}\label{014}
		I_{\lambda}(u) > \delta \text{ for all } u\in E_{\lambda} \text{ with }\left\Vert u\right\Vert_{\lambda }=\rho.
	\end{equation}
	\item[2.] There exists $e\in E_{\lambda}$ with $\left\Vert e\right\Vert_{\lambda}>\rho$, where $\rho >0$ is fixed in \eqref{014}, such that $I_{\lambda}(e)<0$ for all $\lambda >0$.
\end{itemize}
\end{lemma}

\begin{proof} We divide the proof into two steps, one for each different result we are claiming. 

\textbf{Step 1.} Let us start by proving the first mountain pass property. Using the fractional Sobolev inequality and \eqref{06}, for all $u\in E_{\lambda}$, we have
\begin{align}\label{14}
	\int_{\Omega } &\left(\frac{\alpha}{p(x)}\left\vert u\right\vert ^{p(x)}+\frac{\beta}{k(x)}\left\vert u\right\vert^{k(x)}\right) \mathrm{d}x 	
	\\
	&\leq \frac{\alpha }{p^-}\int_{\Omega}\left\vert u\right\vert^{p(x)}\mathrm{d}x+\frac{\beta}{k^-}\int_{\Omega}\left\vert u\right\vert^{k(x)}\mathrm{d}x  \notag
	\\
	&\leq \frac{\alpha}{p^-}\max \left(\C_p^{p^-}\left\Vert u\right\Vert_{\lambda}^{p^-},\C_p^{p^+}\left\Vert u\right\Vert_{\lambda}^{p^+}\right) + \frac{\beta}{k^-}\max \left(\C_k^{k^-}\left\Vert u\right\Vert_{\lambda}^{k^-},\C_k^{k^+}\left\Vert u\right\Vert_{\lambda}^{k^+}\right). \notag
\end{align}
We then get from \eqref{14} that
\begin{align*}
	I_{\lambda}\geq \min \left(\frac{1}{q^+},\frac{1}{2}\right) \left\Vert u\right\Vert_{\lambda}^2 - \frac{\alpha}{p^-}\max \left( \C_p^{p^-}\!,\C_p^{p^+}\right) \left\Vert u\right\Vert_{\lambda}^{p^+} -\frac{\beta}{k^-}\max \left(\C_k^{k^-}\!,\C_k^{k^+}\right) \left\Vert u\right\Vert_{\lambda}^{k^+}
\end{align*}
for all $u\in E_{\lambda}$ with $\left\Vert u\right\Vert_{\lambda}\geq 1$.

Set the constants $A$, $B$ and $C$ as in \eqref{constants} and let $\phi:[0,+\infty)\rightarrow \mathbb{R}$ be an auxiliary
function such that
\begin{equation*}
	\phi (\sigma) =\psi (\sigma) \sigma^{k^+} \text{ for all }\sigma \geq 0,
\end{equation*}
where
\begin{equation*}
	\psi(\sigma) = D\sigma^{2-k^+}-A\alpha\sigma^{p^+-k^+}-B\beta.
\end{equation*}
Let
\begin{align*}
	\sigma^{\ast}:=\left[ \frac{D(2-k^+)}{A\alpha (p^+-k^+)}\right]^{\frac{1}{p^+-2}}.
\end{align*}
We then have
\begin{equation*}
	\psi(\sigma^{\ast}) = \max_{\sigma \geq 0}\psi(\sigma) >0,
\end{equation*}
provided that
\begin{equation*}
	B\beta <\left[\frac{D(2-k^+)}{A\alpha (p^+-k^+)}\right]^{\frac{2-k^+}{p^+-2}}\frac{D(p^+-2)}{(p^+-k^+)},
\end{equation*}
that is,
\begin{equation*}
	\alpha^{2-k^+}\beta^{p^+-2}\leq \left[\frac{D(2-k^+)}{A(p^+-k^+)}\right]^{2-k^+}\left(\frac{D(p^+-2)}{B(p^+-k^+)}\right)^{p^+-2}.
\end{equation*}
Moreover, since we are assuming
\begin{equation*}
	\alpha \leq \frac{D(2-k^+)}{A(p^+-k^+)},
\end{equation*}
we can readily check that $\sigma^\ast\geq 1$. Then, the first mountain pass property holds with $\rho =\sigma^{\ast}>0$ and $\delta =\phi(\sigma ^{\ast })>0$.

\noindent\textbf{Step 2.} Let us now prove the second mountain pass property. In order to do that, let us select $v_0\in E_{\lambda}$ such that
\begin{equation*}
	\left\Vert v_0\right\Vert_{\lambda }=1 \quad\text{ and }\quad \int_{\Omega}\left\vert v_0(x) \right\vert^{p(x)}\mathrm{d}x>0.
\end{equation*}
Then for all $\sigma \geq 1$, we get
\begin{align*}
	I_{\lambda}(\sigma v_0) &\leq \max \left(\frac{1}{q^-},\frac{1}{2}\right)\! \sigma^2\left\Vert v_0\right\Vert_{\lambda}^2-\int_{\Omega}\!\left(\frac{\alpha}{p(x)}\left\vert \sigma v_0(x)\right\vert^{p(x)}+\frac{\beta}{k(x)}\left\vert \sigma v_0(x)\right\vert^{k(x)}\right) \mathrm{d}x
	\\
	&\leq \max \left(\frac{1}{q^-},\frac{1}{2}\right)\!\sigma^2\left\Vert v_0\right\Vert_{\lambda}^2-\frac{\alpha\sigma^{p^-}}{p^+}\int_{\Omega}\left\vert v_0(x)\right\vert^{p(x)}\mathrm{d}x.
\end{align*}

Since $p^->2$, we can chose $\sigma\geq 1$ large enough such that $\left\Vert \sigma v_0\right\Vert_{\lambda}>\rho$ and $I_{\lambda}(\sigma v_0)<0$. Then, the second mountain pass property is satisfied by letting $e=\sigma v_0$.
\end{proof}

Let us now show that the functional $I_\lambda$ provides the $(PS)_c$ condition in $E_\lambda$. To this end, in the same spirit of Theorem \ref{Thes01}, let us define
\begin{equation*}
	c_{\lambda}=\inf_{\gamma \in \Gamma}\max_{0\leq \sigma \leq 1}I_{\lambda}(\gamma(\sigma)),
\end{equation*}
and
\begin{equation*}
	c(\Omega_0) =\inf_{\gamma \in \widetilde{\Gamma}}\;\max_{0\leq \sigma \leq 1}I_{\lambda}\mid_{\mathcal{H}_0^{s(.)}(\Omega_0)}(\gamma(\sigma)),
\end{equation*}
where:
\begin{itemize}
	\item the set $\Omega_0$ is the one provided by assumption \eqref{hypV2};
	\vspace{0.1cm}
	\item $I_{\lambda}\mid_{\mathcal{H}_0^{s(.)}(\Omega_0)}$ is a restriction of $I_{\lambda}$ on $\mathcal{H}_0^{s(.)}(\Omega_0)$;
	\vspace{0.1cm}
	\item $\Gamma =\{\gamma \in C^1([0,1];E_{\lambda}) \;|\; \gamma (0)=1 \text{ and }\gamma(1)=e\}$;
	\vspace{0.1cm}
	\item $\widetilde{\Gamma}=\{\gamma \in C^1([0,1];\mathcal{H}_0^{s(.)}(\Omega_0))\;|\; \gamma (0)=1 \text{ and } \gamma (1)=e\}.$
\end{itemize}

Clearly, $c(\Omega_0)$ is independent of $\lambda$. Moreover, note that, for all $u\in \mathcal{H}_0^{s(.)}(\Omega_0)$, we have
\begin{align*}
	I_{\lambda}\mid_{\mathcal{H}_0^{s(.)}(\Omega_0)}(u) =& \int_{\Omega_0}\int_{\Omega_0}\frac{1}{q(x,y)}\frac{\left\vert u(x)-u(y)\right\vert^{q(x,y)}}{\left\vert x-y\right\vert^{n+q(x,y) s(x,y)}}\,\mathrm{d}x\mathrm{d}y + \frac{\lambda}{2}\int_{\Omega_0}\!V( x)\left\vert u(x)\right\vert^2\mathrm{d}x
	\\
	&-\int_{\Omega_0}\left(\frac{\alpha}{p(x)}\left\vert u\right\vert^{p(x)} + \frac{\beta}{k(x)}\left\vert u\right\vert^{k(x)}\right) \mathrm{d}x,
\end{align*}

By the proof of Lemma \ref{Lemma1}, we can infer that $I_{\lambda}\mid_{\mathcal{H}_0^{s(.)}(\Omega_0)}$ satisfies the mountain pass proprieties of Theorem \ref{Thes01}. Since $\mathcal{H}_0^{s(.)}(\Omega_0) \subset E_{\lambda}$ for all $\lambda >0$, one has $0<\alpha \leq c_{\lambda}\leq c(\Omega_0)$ for all $\lambda >0$. Clearly, for all $\sigma \in [0,1]$, $\sigma e\in \widetilde{\Gamma }$. Thus, there exists $\C_0>0$ such that
\begin{equation}\label{15}
	c(\Omega_0)\leq \max_{0\leq \sigma \leq 1}I_{\lambda}(\sigma e)\leq \C_0< +\infty,
\end{equation}
since $p^->2$. Then,
\begin{equation*}
	0<\delta \leq c_{\lambda }\leq c(\Omega_0)<\C_0
\end{equation*}
for all $\lambda >0$. By Lemma \ref{Lemma1} and Theorem \ref{The1}, we get that for all $\lambda >0$, there exists $\{u_j\}_j\subset E_{\lambda}$ such that
\begin{equation}\label{16}
	I_{\lambda}(u_j)\rightarrow c_{\lambda}>0 \quad\text{ and }\quad I_{\lambda}^{\prime}(u_j)\rightarrow 0, \quad\text{ as }j\rightarrow +\infty.
\end{equation}
Hence, $\{u_j\}_j$ is a $(PS)_{c_\lambda}$ sequence.

\begin{lemma}\label{Lemma3}
Under the assumptions \eqref{hypPbis}, \eqref{hypS2}, \eqref{hypV1}, \eqref{hypV2}, \eqref{hypK1} and \eqref{hypAlpha}, the sequence $\{u_j\}_j$ given by \eqref{16} is bounded in $E_{\lambda}$ for all $\lambda >0$.
\end{lemma}

\begin{proof}
Using the H\"older inequality, by \eqref{hypAlpha}, we obtain
\begin{align}\label{17}
	c_{\lambda} +& \; o(1) \geq I_{\lambda}(u_j)-\frac{1}{p^-} \langle I_{\lambda}^{\prime}(u_j),u_j\rangle \notag
	\\
	=&\iint_{\mathbb{R}^{2n}}\frac{1}{q(x,y)}\frac{\left\vert u_j(x) - u_j(y) \right\vert^{q(x,y)}}{\left\vert x-y\right\vert^{n+q(x,y) s(x,y)}}\, \mathrm{d}x\mathrm{d}y + \frac{\lambda}{2}\int_{\Omega}V(x) \left\vert u_j(x) \right\vert^2\mathrm{d}x - \frac{1}{p^-} \left\Vert u_j\right\Vert_{\lambda}^2  \notag
	\\
	&-\int_{\Omega}\left( \alpha \left(\frac{1}{p(x)} - \frac{1}{p^-}\right) \left\vert u_j(x)\right\vert^{p(x)} + \beta \left(\frac{1}{k(x)}-\frac{1}{p^-}\right) \left\vert u_j(x)\right\vert^{k(x)}\right) \mathrm{d}x  \notag
	\\
	\geq& \left(\min \left(\frac{1}{2},\frac{1}{q^+}\right)-\frac{1}{p^-}\right) \left\Vert u_j\right\Vert_{\lambda}^2 \notag 
	\\
	&-\beta \left(\frac{1}{k^-}-\frac{1}{p^-}\right) \max \left(\C_k^{k^-}\left\Vert u_j\right\Vert_{\lambda}^{k^-},\C_k^{k^+}\left\Vert u_j\right\Vert_{\lambda }^{k^+}\right).
\end{align}

Arguing by contradiction, we assume that $\{u_j\}_j$ is not bounded in $\mathcal{H}_0^{s(.)}(\Omega)$. Then there exists a subsequence, still denoted by $\{u_j\}_j$, such that $\left\Vert u_j\right\Vert_{\lambda}\rightarrow +\infty$ as $j\rightarrow +\infty$. Hence, by \eqref{17}, we have
\begin{equation*}
	\frac{c_{\lambda} + o(1)}{\left\Vert u_j\right\Vert_{\lambda}^2}\geq \left(D-\frac{1}{p^-}\right) -\beta \left(\frac{1}{k^-}-\frac{1}{p^-}\right) \max \left(\C_k^{k^-}\left\Vert u_j\right\Vert_{\lambda}^{k^- -2},\C_k^{k^+}\left\Vert u_j\right\Vert_{\lambda}^{k^+ -2}\right)
\end{equation*}
which yields $2\geq p^-$ or $q^+\geq p^-$. This is a contradiction, thus $\{u_j\}_j$ is bounded in $E_{\lambda}$ for all $\lambda >0$.
\end{proof}

\begin{lemma}\label{Lemma4}
Assume that \eqref{hypPbis}, \eqref{hypS2}, \eqref{hypV1}-\eqref{hypV2}, \eqref{hypK1} and \eqref{hypAlpha} hold. Then $I_{\lambda}$ satisfies the $(PS)_c$ condition in $E_{\lambda}$ for all $c\in\mathbb{R}$ and $\lambda >0$.
\end{lemma}

\begin{proof}
Let $\{u_j\}_j$ be a $(PS)_c$ sequence with $c<\C_0$, where $\C_0$ is te constant introduced in \eqref{15}. By Lemma \ref{Lemma3}, $\{u_j\}_j$ is bounded in $E_{\lambda}$ and there exists $\C>0$ such that $\left\Vert u_j\right\Vert_{\lambda}\leq \C$. Thus, there exist a subsequence of $\{u_j\}_j$, still denoted by $\{u_j\}_j$, and $u_0$ in $E_{\lambda}$ such that as $j\to +\infty$
\begin{equation*}
	\begin{array}{ll}
		u_j \rightharpoonup u_0 &\text{weakly in } E_{\lambda},
		\\[4pt]
		u_j \rightharpoonup u_0 &\text{a.e. in }\mathbb{R},
		\\
		\left\vert u_j\right\vert^{p(.)-2}u_j \rightharpoonup \left\vert u_0\right\vert^{p(.)-2}u_0 & \text{weakly in }L^{\frac{p(.)}{p(.)-1}}(\Omega)
	\end{array}
\end{equation*}

Our aim now is to prove that $u_j\rightarrow u_0$ strongly in $E_{\lambda}$. By Proposition \ref{Thes2}, we obtain $u_j\rightarrow u_0$ in $L^{p(.)}(\Omega)$ and $L^{k(.)}(\Omega)$, respectively. Thus
\begin{equation*}
	\lim_{j\rightarrow +\infty}\int_{\Omega}\left\vert u_j(x)-u_0\right\vert^{p(x)}\mathrm{d}x=0
\end{equation*}
and
\begin{equation*}
	\lim_{j\rightarrow +\infty} \int_{\Omega}\left\vert u_j(x) - u_0\right\vert^{k(x)}\mathrm{d}x=0.
\end{equation*}
It follows from \eqref{16} that
\begin{align*}
	o(1) =&\; \langle I_{\lambda}^{\prime}(u_j)-I_{\lambda}^{\prime}(u_0),u_j-u_0\rangle
	\\
	=&\; (u_j-u_0,u_j-u_0)_{\lambda}-\alpha \int_{\Omega}\left(\left\vert u_j\right\vert^{p(x)-2}u_j-\left\vert u_0\right\vert^{p(x)-2}u_0\right)(u_j-u_0)\,\mathrm{d}x
	\\
	&-\beta \int_{\Omega}\left(\left\vert u_j\right\vert^{k(x)-2}u_j-\left\vert u_0\right\vert^{k(x)-2}u_0\right)(u_j-u_0)\,\mathrm{d}x,
\end{align*}
which means that
\begin{equation*}
	\lim_{j\rightarrow +\infty}\left\Vert u_j-u_0\right\Vert_{\lambda}=0.
\end{equation*}
This completes the proof.
\end{proof}

\begin{proof}[Proof of Theorem \ref{The1}]
	
First of all, from Lemmas \ref{Lemma1}-\ref{Lemma3} and Theorem \ref{Thes01}, we deduce that for all $\lambda>0$ there exists a $(PS)_{\alpha\lambda}$ sequence $\{u_j\}_j$ for $I_{\lambda}$ on $E_{\lambda}$.

Now, by Lemma \ref{Lemma3} and the fact that $0<c_{\lambda}<c(\Omega_0)<\C_0$ for all $\lambda >0$, where $\C_0$ is the constant introduced in \eqref{15}, we obtain that there exists a subsequence of $\{u_j\}_j$, still denoted by $\{u_j\}_j$, and $u_{\lambda}^{(1)}\in E_{\lambda}$ such that $u_j\rightarrow u_{\lambda}^{(1)}$ strongly in $E_{\lambda}$. Moreover, $I_{\lambda}(u_j) = c_{\lambda}\geq\delta$ and $u_{\lambda}^{(1)}$ is a solution of \eqref{mainProblem}.

Next, we show that system \eqref{mainProblem} has another solution. For this purpose, let us define
\begin{equation*}
	\widetilde{c}_{\lambda}:=\inf \left\{I_{\lambda}(u)\;|\;u\in \overline{B}_{\rho}\right\},
\end{equation*}
where $B_{\rho}=\{u\in E_{\lambda}\;|\;\left\Vert u\right\Vert_{\lambda}<\rho\}$ and $\rho >0$ is given by Lemma \ref{Lemma1}. Moreover, let $w_0\in \mathcal{H}_0^{s(.)}(\Omega)\subset E_\lambda$ be such that
\begin{align*}
	\int_{\Omega}\left\vert w_0\right\vert^{k(x)}\,\mathrm{d}x>0.
\end{align*}
We can readily check that
\begin{equation*}
	I_{\lambda}(\tau w_0)\leq \tau^2\max \left(\frac{1}{q^-},\frac{1}{2}\right) \left\Vert w_0\right\Vert_{\lambda}^2-\frac{\beta \tau^{k^+}}{k^+}\int_{\Omega}\left\vert w_0\right\vert^{k(x)}\mathrm{d}x.
\end{equation*}
Hence, since by \eqref{hypK1} we have $k^+<2$, for all $\tau <\tau_0$ with
\begin{align*}
	\tau_0 := \left[\frac{\beta}{k^+ \left\Vert w_0\right\Vert_{\lambda}^2}\left(\int_{\Omega}\left\vert w_0\right\vert^{k(x)}\,\mathrm{d}x\right)\left(\max \left(\frac{1}{q^-},\frac{1}{2}\right)\right)^{-1}\right]^{\frac{1}{2-k^+}}
\end{align*}
we immediately have $I_{\lambda}(\tau w_0)<0$ for all $\lambda>0$.

On the other hand, by taking $\tau\leq \rho\left\Vert w_0\right\Vert_{\lambda}^{-1}$, we also have that $\tau w_0\in B_{\rho}$. Hence, if $\tau\leq\min(\tau_0,\rho\left\Vert w_0\right\Vert_{\lambda}^{-1})$, there exists $z_0=\tau w_0\in B_{\rho}$ such that $I_{\lambda}(z_0)<0$ for all $\lambda >0$. This clearly implies that $\widetilde{c}_{\lambda}<0$ for all $\lambda >0$.

It then follows from Lemma \ref{Lemma1} and the Ekeland variational principle (see \cite{Ekeland}) applied in $B_{\rho}$, that there exists a sequence $\{u_j\}_j$ such that
\begin{equation}\label{21}
	\widetilde{c}_{\lambda}\leq I_{\lambda}(u_j)\leq \widetilde{c}_{\lambda}+\frac{1}{j},
\end{equation}
and
\begin{equation}\label{22}
	I_{\lambda }(v)\geq I_{\lambda}(u_j)-\frac{1}{j}\left\Vert u_j-v\right\Vert_{\lambda}
\end{equation}
for all $v\in B_{\rho}$. 

Now we show that $\left\Vert u_j\right\Vert_{\lambda}<\rho$ for $n$ sufficiently large. Arguing by contradiction, we
assume that $\left\Vert u_j\right\Vert_{\lambda }=\rho $ for infinitely many $j$. Without loss of generality, we may assume that $\left\Vert u_j\right\Vert_{\lambda}=\rho$ for any $j\in \mathbb{N}$. From Lemma \ref{Lemma1}, we deduce that
\begin{equation*}
	I_{\lambda}(u_j)\geq \delta >0.
\end{equation*}

This, combined with \eqref{21}, implies that $\widetilde{c}_{\lambda}\geq \delta >0$, which contradicts $\widetilde{c}_{\lambda}<0$. Next we show that $I_{\lambda}^{\prime}(u_j)\rightarrow 0$ in $E_{\lambda}^{\ast}$. Set
\begin{equation*}
	w_j = u_j +\tau v, \text{ for all } v\in B_1:=\{v\in E_{\lambda}\;|\;\left\Vert v\right\Vert_{\lambda}=1\},
\end{equation*}
where $\tau >0$ small enough is such that $2\tau \rho +\tau^2\leq \rho^2-\left\Vert u_j\right\Vert_{\lambda}^2$ for fixed $j$ large. Then
\begin{equation*}
	\left\Vert w_j\right\Vert_{\lambda}^2 = \left\Vert u_j\right\Vert_{\lambda}^2 + 2\tau \rho \langle u_j,v\rangle_{\lambda} + \tau^2\leq \left\Vert u_j\right\Vert_{\lambda}^2+2\tau \rho +\tau^2\leq \rho^2,
\end{equation*}
which implies that $w_j\in B_{\rho}$. Thus, from \eqref{22}, we obtain
\begin{equation*}
	I_{\lambda}(w_j)\geq I_{\lambda}(u_j)-\frac{1}{j}\left\Vert u_j-w_j\right\Vert_{\lambda},
\end{equation*}
that is
\begin{equation*}
	\frac{I_{\lambda}(u_j+\tau v)-I_{\lambda}(u_j)}{\tau}\geq -\frac{1}{j}.
\end{equation*}

Letting $\tau \rightarrow 0^+$, we get $\langle I_{\lambda}^{\prime}(u_j),v\rangle \geq -1/j$ for any fixed $j$ large. Similarly, choosing $\tau <0$ and $\left\vert \tau \right\vert$ small enough and repeating the procedure above, one can obtain that $\langle I_{\lambda}^{\prime}(u_j),v\rangle \leq 1/j$ for any fixed $j$ large. Thus,
\begin{equation*}
	\lim_{j\rightarrow +\infty}\sup_{v\in B_1}\left\vert \langle I_{\lambda}^{\prime}(u_j),v\rangle\right\vert =0,
\end{equation*}
which yields that $I_{\lambda}(u_j)\rightarrow 0$ in $E_{\lambda}^{\ast}$ as $j\rightarrow +\infty$. Therefore, $\{u_j\}_j$ is a $(PS)_{\widetilde{c}_{\lambda}}$ sequence for the functional $I_{\lambda}$. Using a similar discussion as in Lemma \ref{Lemma4}, there exists $u_{\lambda}^{(2)}\in E_{\lambda}$ such that $u_j\rightarrow u_{\lambda}^{(2)}$ in $E_{\lambda}$. Thus, we get a nontrivial solution $u^{(2)}$ of \eqref{mainProblem} satisfying
\begin{equation*}
	I_{\lambda}(u_{\lambda}^{(2)})\leq \zeta <0 \quad\text{ and }\quad \left\Vert u_{\lambda}^{(2)}\right\Vert_{\lambda}<\rho.
\end{equation*}
We therefore deduce that
\begin{equation*}
	I_{\lambda }(u_{\lambda}^{(2)})=\widetilde{c}_{\lambda}\leq \zeta <0<\delta <c_{\lambda}=I_{\lambda}(u_{\lambda}^{(1)}) \quad\text{ for all } \lambda >0,
\end{equation*}
which ends the proof.
\end{proof}

\section{Proof of Theorems \ref{The2} and \ref{The3}}\label{sec:proof2}

We give in this section the proof of the others two main results of this paper, namely Theorems \ref{The2} and \ref{The3}.

\begin{proof}[Proof of Theorem \ref{The2}]
For any sequence $\{\lambda_j\}_j$ such that $1\leq \lambda_j\rightarrow +\infty$ as $j\rightarrow +\infty$, let $u_j^{(i)}$ be the critical points of $I_{\lambda}$ obtained in Theorem \ref{The1} for $i=1,2$. Thus, we have
\begin{equation*}
	I_{\lambda}(u_{\lambda}^{(2)})\leq \zeta <0 \quad\text{ and }\quad \left\Vert u_{\lambda}^{(2)}\right\Vert_{\lambda }<\rho.
\end{equation*}
Hence, we deduce that
\begin{align}\label{25}
	&I_{\lambda_j}(u_j^{(2)})\leq \zeta <0<\delta <c_{\lambda_j}=I_{\lambda_j}(u_j^{(1)})<\C_0,
	\\
	&I_{\lambda_j}^{\prime}(u_j^{(2)})=I_{\lambda_j}^{\prime}(u_j^{(1)}), \notag
\end{align}
where $\C_0$ is the constant introduced in \eqref{15}, and
\begin{align}\label{26}
	I_{\lambda_j}\left(u_j^{(i)}\right) \geq&\; \min \left(\frac{1}{2},\frac{1}{q^+}\right) \left\Vert u_j^{(i)}\right\Vert_{\lambda_j}^2 -\int_{\Omega}\left(\frac{\alpha}{p(x)}\left\vert u_j^{(i)}\right\vert^{p(x)} + \frac{\beta}{k(x)}\left\vert u_j^{(i)}\right\vert^{k(x)}\right) \mathrm{d}x \notag
	\\
	\geq&\; \left(\min \left( \frac{1}{2},\frac{1}{q^+}\right) -\frac{1}{p^-}\right) \left\Vert u_j^{(i)}\right\Vert_{\lambda_j}^2 \notag
	\\
	&-\beta \left(\frac{1}{k^-}-\frac{1}{p^-}\right) \max \left(\C_k^{k^-}\left\Vert u_j^{(i)}\right\Vert_{\lambda_j}^{k^-},\C_k^{k^+}\left\Vert u_j^{(i)}\right\Vert_{\lambda_j}^{k^+}\right).
\end{align}
We then get from \eqref{25} and \eqref{26} that
\begin{equation*}
	\left\Vert u_j^{(i)}\right\Vert_{\lambda_j}\leq \C,
\end{equation*}
where $\C>0$ is independent of $\lambda_j$. Hence, we can suppose that $u_j^{(i)}\rightharpoonup u^{(i)}$ weakly in $\mathcal{H}_0^{s(.)}(\Omega)$ and $u_j^{(i)}\rightarrow u^{(i)}$ strongly in $L^{p(.)}(\Omega)$ and $L^{k(.)}(\Omega)$, respectively. By Fatous' lemma, we obtain
\begin{equation*}
	\int_{\Omega}V(x) \left\vert u^{(i)}(x) \right\vert^2\mathrm{d}x \leq \liminf_{j\rightarrow +\infty}\int_{\Omega}V(x) \left\vert u_j^{(i)}(x) \right\vert^2\mathrm{d}x \leq \liminf_{j\rightarrow +\infty}\frac{\left\Vert u_j^{(i)}\right\Vert_{\lambda_j}^2}{\lambda_j}=0.
\end{equation*}
Thus, $u^{(i)}=0$ a.e. in $\mathbb{R}^n\backslash V^{-1}(0)$ and $u^{(i)}\in \mathcal{H}_0^{s(.)}(\Omega_0)$ by \eqref{hypK1}. 

Similarly to the proof of Theorem \ref{The1}, we can now prove that $u^{(1)}$ and $u^{(2)}$ are two solutions of problem \eqref{23}. Indeed, it follows from \eqref{25}, $u^{(i)}=0$ a.e. in $\mathbb{R}^n\backslash V^{-1}(0)$ and the constants $\zeta,\delta$ are independent of $\lambda$ that
\begin{equation*}
	\max \left(\frac{1}{q^-},\frac{1}{2}\right) \left\Vert u^{(1)}\right\Vert_{\lambda}^2-\int_{\Omega_0}\frac{\alpha}{p(x)}\left\vert u^{(1)}\right\vert^{p(x)} \mathrm{d}x - \int_{\Omega_0}\frac{\beta}{k(x)}\left\vert u^{(1)}\right\vert^{k(x)}\mathrm{d}x\geq \delta >0
\end{equation*}
and
\begin{equation*}
	\min \left(\frac{1}{q^+},\frac{1}{2}\right) \left\Vert u^{(2)}\right\Vert_{\lambda}^2-\int_{\Omega_0}\frac{\alpha}{p(x)}\left\vert u^{(2)}\right\vert^{p(x)} \mathrm{d}x-\int_{\Omega_0}\frac{\beta}{k(x)}\left\vert u^{(2)}\right\vert^{k(x)}\mathrm{d}x\leq \zeta <0
\end{equation*}
which means that $u^i\neq 0$ and $u^1\neq u^2$. The proof is thus complete.
\end{proof}

We conclude this paper with the proof of Theorem \ref{The3}, providing the existence of infinitely many solutions of problem \eqref{23}. To this end, we will employ the following symmetric mountain pass theorem (see \cite[Theorem 2.2]{Colasuonno}).

\begin{theorem}\label{The4}
Let $X$ be a real infinite dimensional Banach space and $J\in C^1(X)$ a functional satisfying the $(PS)_c$ condition as well as the
following three properties:
\begin{itemize}
	\item[1.] $J(0)=0$ and there exist two constants $\rho,\delta >0$ such that $J(u)\geq \delta$ for all $u\in X$ with $\left\Vert u\right\Vert=\rho$.

	\item[2.] $J$ is even.

	\item[3.] For all finite dimensional subspaces $Y\subset X$ there exists $R=R(Y)>0$ such that $J(u)\leq 0$ for all $u\in X\backslash B_R(Y)$, where $B_R(Y)=\{u\in Y\;|\;\left\Vert u\right\Vert \leq R\}$.
\end{itemize}
Then $J$ possesses an unbounded sequence of critical values characterized by a mini-max argument.
\end{theorem}

\begin{proof}[Proof of Theorem \protect\ref{The3}]

Obviously, Lemma \ref{Lemma1} still holds when considering functions $u\in \mathcal{H}_0^{s(.)}(\Omega_0)$. Let us now define the functional $I:\mathcal{H}_0^{s(.)}(\Omega_0) \rightarrow \mathbb{R}$ as
\begin{equation*}
	I(u) \leq \max \left(\frac{1}{q^-},\frac{1}{2}\right) \left\Vert u\right\Vert_{\lambda}^2-\int_{\Omega_0}\frac{\alpha}{p(x)}\left\vert u\right\vert^{p(x)}\mathrm{d}x-\int_{\Omega_0}\frac{\beta}{k(x)}\left\vert u\right\vert^{k(x)}\mathrm{d}x.
\end{equation*}

Clearly, $I\in C^1(\mathcal{H}_0^{s(.)}(\Omega_0))$ and the critical points of $I$ are the weak solutions of problem \eqref{23}.

Now we first due that for any finite dimensional subspace $\mathcal{W}$ of $\mathcal{H}_0^{s(.)}(\Omega_0)$, there exists $R_0=R_0(\mathcal{W})$ such that $I(u)<0$ for all $u\in \mathcal{H}_0^{s(.)}(\Omega_0) \backslash B_{R_0}(\mathcal{W})$, where $B_{R_0}(\mathcal{W})=\{u\in \mathcal{H}_0^{s(.)}(\Omega_0) \;|\; \left\Vert u\right\Vert_{\lambda}<R_0\}$.

Next, let $\mathcal{W}$ be a fixed finite dimensional subspace of $\mathcal{H}_0^{s(.)}(\Omega_0)$, for any $u\in \mathcal{W}$ such that $\left\Vert u\right\Vert_{\lambda}^2=1$. Thus, we get that for all $\sigma \geq 1$,
\begin{align*}
	I(\sigma u) &\leq \max \left( \frac{1}{q^-},\frac{1}{2}\right) \sigma^2\left\Vert u\right\Vert_{\lambda}^2-\int_{\Omega_0}\frac{\alpha}{p(x)}\left\vert \sigma u\right\vert^{p(x)}\mathrm{d}x-\int_{\Omega_0}\frac{\beta}{k(x)}\left\vert \sigma u\right\vert^{k(x)}\mathrm{d}x
	\\
	&\leq \max \left(\frac{1}{q^-},\frac{1}{2}\right) \sigma^2\left\Vert u\right\Vert_{\lambda}^2-\frac{\alpha \sigma^{p^-}}{p^+}\min\left( \left\Vert u\right\Vert_{L^{p(.)}(\Omega_0)}^{p^-},\left\Vert u\right\Vert_{L^{p(.)}(\Omega_0)}^{p^+}\right).
\end{align*}

Note that there exists $\C_{\mathcal{W}}>0$ such that $\left\Vert u\right\Vert_{L^{p(.)}(\Omega_0)}\geq \C_{\mathcal{W}}\left\Vert u\right\Vert_{\lambda}$, because all norms are equivalent on the finite dimensional Banach space $\mathcal{W}$. Hence, since $p^->2$, we get
\begin{equation*}
	I(\sigma u) \leq \max \left(\frac{1}{q^-},\frac{1}{2}\right)\sigma^2\left\Vert u\right\Vert_{\lambda}^2-\frac{\alpha \sigma^{p^-}}{p^+}\min \left(\C_{\mathcal{W}}^{p^-},\C_{\mathcal{W}}^{p^+}\right) \rightarrow -\infty \quad\text{ as }\sigma \rightarrow +\infty.
\end{equation*}
Thus, as $R\rightarrow +\infty$,
\begin{equation*}
	\sup_{\underset{\left\Vert u\right\Vert_{\lambda}=R}{u\in \mathcal{H}_0^{s(.)}(\Omega_0)}} I(u) = \sup_{\underset{\left\Vert u\right\Vert_{\lambda}=1}{u\in \mathcal{H}_0^{s(.)}(\Omega_0)}} I(Ru)\rightarrow -\infty.
\end{equation*}

Therefore, there exists $R_0>0$ large enough such that $I(u)<0$ for all $u\in \mathcal{H}_0^{s(.)}(\Omega_0)$, with $\left\Vert u\right\Vert_{\lambda }=R$ and $R\geq R_0$. Thus the claim holds true.

Similarly to the proof of Lemma \ref{Lemma3}, one can show that $I$ satisfies the $(PS)_c$ condition for any $c\in\mathbb{R}$. Obviously, $I(0)=0$ and $I$ is an even functional. In conclusion, by Theorem \ref{The4}, there exists an unbounded sequence of solutions of problem \eqref{23}.
\end{proof}

\section{Conclusions and open problems}\label{sec:open_pb}

In this paper, we have considered the following second-order non-local elliptic equation with variable growth conditions driven by the variable-order fractional Laplace operator:
\begin{align*}
	\begin{cases} 
		(-\Delta)_{q(.)}^{s(.)}u +\lambda Vu = \alpha \left\vert u\right\vert^{p(.)-2}u+\beta \left\vert u\right\vert^{k(.)-2}u & \text{ in }\Omega 
		\\[7pt]
		u=0 & \text{ in }\mathbb{R}^n\backslash\Omega 
	\end{cases}
\end{align*}

Under suitable assumptions for the functions $q(.)$, $s(.)$, $V(.)$, $p(.)$ and $k(.)$, and on the parameters $\alpha$ and $\beta$, we employed the mountain pass category theorem and Ekeland's variational principle to obtain the existence of a least two different solutions for all $\lambda>0$. Moreover, we proved that these solutions converge to two solutions of a limit problem as $\lambda \rightarrow +\infty $, and we obtained the existence of infinitely many solutions for this limit problem. Our results generalize the ones previously obtained in \cite{Mingqi} in the case $q(.)\equiv q$ constant.

We now conclude this paper by presenting a small collection of open problems related to our work which may be of interest for future research.

\begin{itemize}
	\item[1.] A crucial assumptions for our results was that
	\begin{equation*}
		q^{\ast}(x):= \frac{nq(x,y)}{n-s(x,y)q(x,y)}> r(x)
	\end{equation*}
	for all $x\in\overline{\Omega}$ and some continuous function $r$ such that $r(x)>1$. It would be interesting to analyze what results we can get in the case in which $q^\ast(x_0) = r(x_0)$ for some $x_0\in \Omega$.
	
	\item[2.] It would be worth to investigate what happens if in \eqref{mainProblem} we replace the Dirichlet homogeneous exterior condition $u=0$ in $\mathbb{R}^n\backslash \Omega $ with the non-homogeneous one $u=g$ in $\mathbb{R}^{n}\backslash \Omega$, where $g$ satisfies some regularity property to be defined. Notice that this would correspond to a non-homogeneous boundary condition on $\partial\Omega$ in the local case of the Laplace operator.
	
	\item[3.] A final interesting open problem would be to study the existence and multiplicity of solutions for Schr\"odinger-Kirchhoff type models involving the operator $(-\Delta)_{q(.)}^{s(.)}$ with variable exponent nonlinearities, i.e. 
	\begin{equation*}
		\begin{cases}
			\mathcal M (-\Delta)_{q(.)}^{s(.)} u + \lambda Vu=\alpha \left\vert u\right\vert^{p(.)-2}u + \beta \left\vert u\right\vert^{k(.)-2}u & \text{ in }\Omega, 
			\\[7pt]
			u=0 & \text{ in }\mathbb{R}^n\backslash \Omega.
		\end{cases}		
	\end{equation*}
	with $\mathcal M = \mathcal M([u]_{q(.),s(.),\Omega})$. Such kind of problems have been recently analyzed in \cite{xiang} for the constant-exponent fractional $q$-Laplacian. Nevertheless, to the best of our knowledge, the case of the variable-exponent operator remains open.
\end{itemize}

\end{document}